\documentclass[12pt, twoside, leqno]{article}
\usepackage{amsmath,amsthm}
\usepackage{amssymb}
\usepackage[table]{xcolor}

\usepackage{enumerate}

\usepackage{graphicx}
\usepackage{epstopdf}

\makeatletter
\@namedef{subjclassname@2010}{%
  \textup{2010} Mathematics Subject Classification}
\makeatother

\newtheorem{thm}{Theorem}[section]

\newtheorem{prop}[thm]{Proposition}

\theoremstyle{definition}

\newtheorem{rem}[thm]{Remark}

\numberwithin{equation}{section}

\begin{document}

%%%%% To ease editing, for IMPAN journals add:

\baselineskip=17pt

%%%%%%%%%%%

%% In the running head, replace first names by initials 
%% and give an abbreviation of the title.

%\title{Some properties of binomial coefficients and their application to growth modeling}
\title{Use of difference tables at studying of properties of binomial coefficients}

%\author[V. Gavrikov]{Vladimir Gavrikov}
%\address{Siberian Federal University\\
%Svobodny pr. 79, 660041 Krasnoyarsk, Russia}
%\email{vgavrikov@sfu-kras.ru}

\author{Vladimir L. Gavrikov\\
Siberian Federal University\\ 
Svobodny pr. 79, 660041 Krasnoyarsk, Russia\\
E-mail: vgavrikov@sfu-kras.ru}

\date{}

\maketitle

\begin{abstract}
Some properties of diagonal binomial coefficients were studied in respect to frequency of their units digits. An approach was formulated that led to use of difference tables to predict if certain units digits can appear in the values of binomial coefficients at quadratic terms of the binomial theorem. Frequency distributions of units digits of binomial coefficients contain gaps (zero frequency) under most common numbering systems with supposed exclusion to systems with $2^n$ bases. Binomial coefficient arithmetics may be used to model cell population dynamics in a multicellular organism, which implies that the dynamics obeys power function laws.\\


\end{abstract}

%\subjclass[2010]{Primary 11Axx; Secondary 62P10}

%\keywords{natural sequence, minor total, numeral systems, binomial coefficients, multicellular growth}

\section{Introduction}

Since long ago binomial coefficients are known to appear in the binomial theorem of power decomposition of
\begin{equation}
(1 + x)^n = 1 + nx + C_2\cdot x^2 + C_3\cdot x^3 + \dots,
\label{eq:1}
\end{equation}
$x$ being a real variable, $C_2$, $C_3$ etc. are binomial coefficients at quadratic, cubic etc. terms, correspondingly. One of the most vivid representation of the coefficients is Pascal triangle (fig. \ref{fig1}) which is a triangular table where every $n_{th}$ row is the coefficients of the decomposition.

\begin{figure}[tb]
\begin{minipage}[h]{0.49\textwidth}
\center{\includegraphics[width=1.0\textwidth]{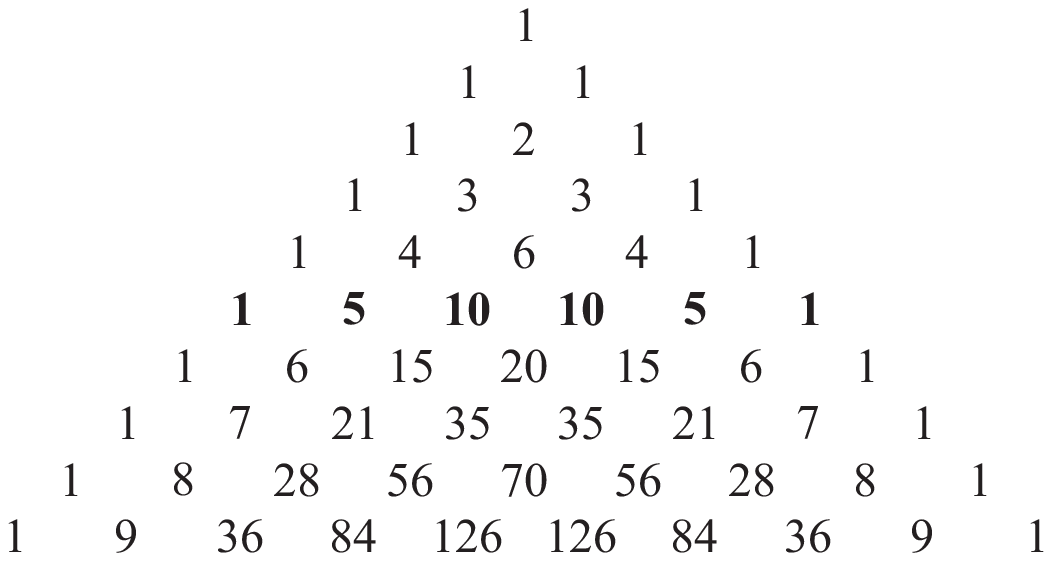} \\ a}
\end{minipage}
\hfill
\begin{minipage}[h]{0.49\textwidth}
\center{\includegraphics[width=1.0\textwidth]{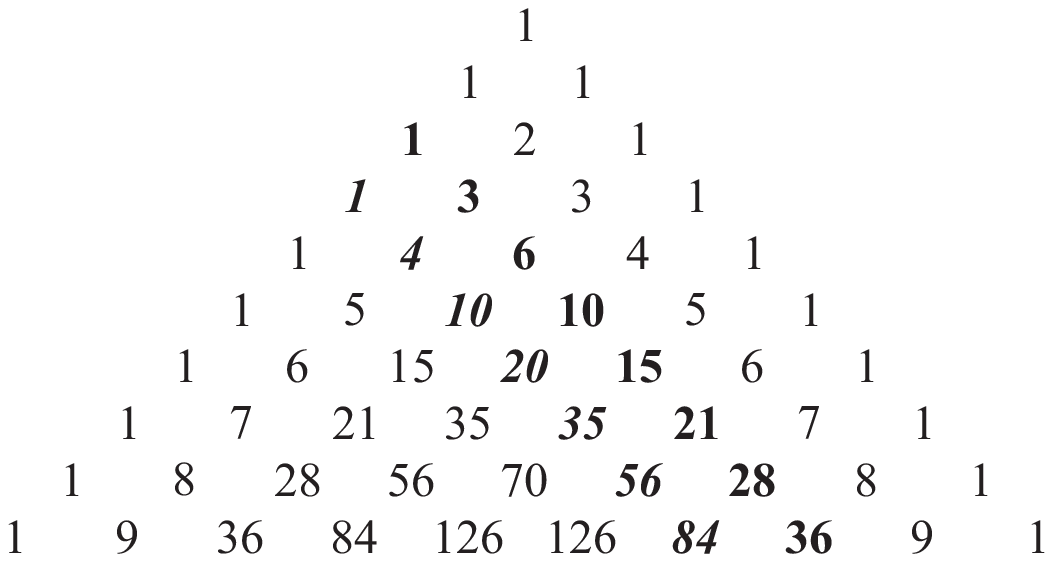} \\ b}
\end{minipage}
\caption{Pascal triangle. \textbf{a} -- A horizontal row with constant $n$ ($n = 5$), as an example, is highlighted in bold face; \textbf{b} -- diagonals of the triangle are highlighted in which $n$ is a variable. Minor totals of natural sequence (binomial coefficients at the quadratic term) are given in bold face. Binomial coefficients at the cubic term are given in bold italic face.}
\label{fig1}
\end{figure}

Arithmetical properties of binomial coefficients are well known and relate mostly to the divisibility by prime numbers and their degrees \cite{Winberg} (Lucas, Kummer theorems etc.) as well as to sums of the coefficients in one row, i.e. where $n$ is a constant (fig. \ref{fig1}a).

It has been shown \cite{Gavrikov} that some subsets of the coefficients from diagonals of the Pascal triangle (fig. \ref{fig1}b) can possess other properties as well. Particularly, the frequency distributions of units digits of binomial coefficients at quadratic terms (which are minor totals of natural sequence) may or may not have gaps dependently on numbering system considered. This property was strictly proved on the example of base-three and base-four numbering systems. In terms of modular arithmetics the properties look as follows, $k$ being a whole number. So, $S_{3k+0} \equiv 0$ (mod $3$), $S_{3k+1} \equiv 1$ (mod $3$), $S_{3k+2} \equiv 0$ (mod $3$) while $S_n \not\equiv 2$ (mod $3$) at any $n$. Then, $S_{4k+0} \equiv 0$ (mod $4$) at even $k$, $S_{4k+0} \equiv 2$ (mod $4$) at odd $k$, $S_{4k+1} \equiv 1$ (mod $4$) at even $k$, $S_{4k+1} \equiv 3$ (mod $4$) at odd $k$. 

An approach was also suggested of how to prove presence or absence of gaps in frequency distributions of units digits of minor totals of natural sequence under numbering system with other bases. Empirically, the gaps are found in numbering systems with bases $5$, $6$, $7$, $9$, $10$ while there are no gaps in systems with bases $4$, $8$, $16$. For example, empirical frequency distributions of units digits for base-seven and base-eight numbering systems are given in fig.\ref{fig2}.

\begin{figure}[tb]
\begin{minipage}[h]{0.49\textwidth}
\center{\includegraphics[width=1.0\textwidth]{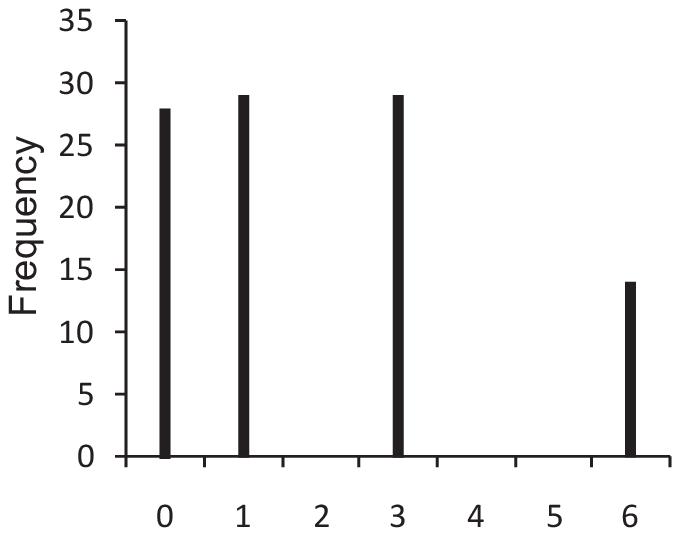} \\ a}
\end{minipage}
\hfill
\begin{minipage}[h]{0.49\textwidth}
\center{\includegraphics[width=1.0\textwidth]{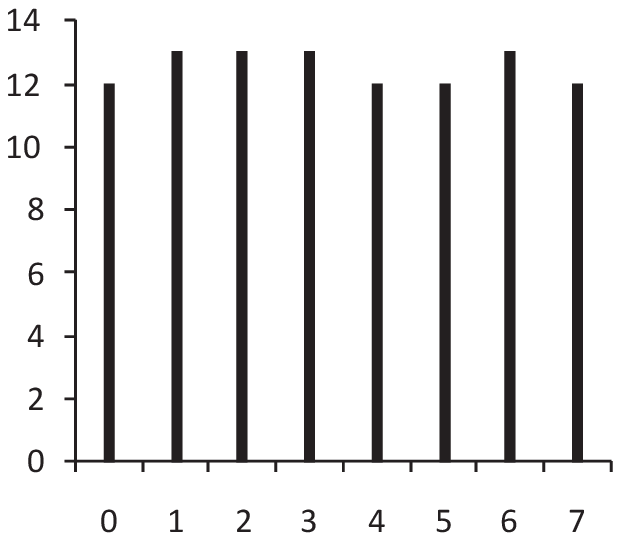} \\ b}
\end{minipage}
\caption{Empirical frequency distributions of units digits in binomial coefficients at the quadratic term (fig.\ref{fig1}b). \textbf{a} -- Base-seven numbering system; \textbf{b} -- base-eight numbering system. Numbers at abscissa axis are digits of the corresponding numbering systems. Frequency is the number of the cases among first hundred of binomial coefficients values.}
\label{fig2}
\end{figure}

\section{Generalization}

Because binomial coefficients at the quadratic term are minor totals of natural sequence (fig. \ref{fig1}b) the formula to calculate them is required which has been known since long ego. The $n_{th}$ minor total of the sequence $S_n$ is
\begin{equation}
S_n = \sum^n_{\kappa =1}\kappa = \frac{n(n+1)}{2}.
\label{eq:2}
\end{equation}

The basic equation of the analysis looks like
\begin{equation}
S_{Lk+i} = L\cdot m + j,
\label{eq:3}
\end{equation}
where $L$ is the base of a numbering system, $i, j$ are units digits in different representations of minor totals $S_n$, obviously $0 \leq  i, j \leq (L-1)$. Values of $k$ and $m$ are whole numbers, which is of core importance.

According to relation \ref{eq:2}
\begin{equation}
S_{Lk+i} = \frac{L\cdot L\cdot k^2 + L\cdot k(2i + 1) + i(i +1)}{2}.
\label{eq:4}
\end{equation}

Then the generalizing relation linking $m$, $k$, $i$ and $j$ can be received from equations \ref{eq:3} and \ref{eq:4} by expressing of $m$ through all other parameters:
\begin{eqnarray}
m& = & \nonumber \\
&= & \frac{k(L\cdot k + (2i + 1))}{2} + \label{eq:5_1}\\
&+ &\frac{\cfrac{i(i +1)}{2}- j}{L}. \label{eq:5_2}
\end{eqnarray}

The wholeness of m depends on evenness/oddness of term $k(L\cdot k + (2i + 1))$ and on whether term $\cfrac{i(i +1)}{2}- j$ is divisible by $L$ with or without a remainder. At given $k$ and $L$, those combinations of $i$ and $j$ that ensure wholeness of $m$ determine $S_{Lk+i} \equiv j \ (mod \ L)$.

Thus the generalizing relation \ref{eq:5_1}--\ref{eq:5_2} can be used to predict the units digits in values of binomial coefficients at quadratic terms ($S_n$) at a given base of numbering system $L$.

Properties of term \ref{eq:5_1} depend on evenness/oddness of $k$ and $L$. If $L$ is odd then term \ref{eq:5_1} is always a whole number ($k(L\cdot k + (2i + 1)) \equiv 0 \ (mod \ 2)$). If $L$ is even then at even $k$ term \ref{eq:5_1} is a whole number but at odd $k$ the term \ref{eq:5_1} is a fraction: ($k(L\cdot k + (2i + 1)) \equiv \cfrac{1}{2} \ (mod \ 2)$).

It is easy to see that $\cfrac{i(i +1)}{2}- j$ brings about a table (Table \ref{tab:1}) that may called a 'difference table' with contains all the possible differences between $\cfrac{i(i +1)}{2}$ and $j$. This table can be used to find the sought quantities of $j$.

\begin{table}[htb]
  \centering
  \caption{A fragment of difference table. First eight values of $\cfrac{i(i +1)}{2}$ and $j$ are shown.}
    \begin{tabular}{rrrrrrrrrrc}
    \hline
    j&&\multicolumn{8}{c}{$\cfrac{i(i +1)}{2}$}\\ \cline{3-11}
     &    & 0     & 1     & 3     & 6     & 10    & 15    & 21    & 28 &$\cdots$ \\
    \hline
    0 &   & 0     & 1     & 3     & 6     & 10    & 15    & 21    & 28 &$\cdots$ \\
    1 &   & -1    & 0     & 2     & 5     & 9     & 14    & 20    & 27 &$\cdots$ \\
    2 &   & -2    & -1    & 1     & 4     & 8     & 13    & 19    & 26 &$\cdots$ \\
    3 &   & -3    & -2    & 0     & 3     & 7     & 12    & 18    & 25 &$\cdots$ \\
    4 &   & -4    & -3    & -1    & 2     & 6     & 11    & 17    & 24 &$\cdots$ \\
    5 &   & -5    & -4    & -2    & 1     & 5     & 10    & 16    & 23 &$\cdots$ \\
    6 &   & -6    & -5    & -3    & 0     & 4     & 9     & 15    & 22 &$\cdots$ \\
    7 &   & -7    & -6    & -4    & -1    & 3     & 8     & 14    & 21 &$\cdots$ \\
$\cdots$& &$\cdots$ &$\cdots$ &$\cdots$ &$\cdots$ &$\cdots$ &$\cdots$ &$\cdots$ &$\cdots$ &$\cdots$ \\
    \hline
    \end{tabular}%
  \label{tab:1}%
\end{table}%

\begin{rem}
The difference table may be used to predict not only units digits in $S_n$ but other digits as well. For example, considering $S_{L^2k+i} = L^2\cdot m + j$ one can study appearance of tens-and-units digits of base-$L$ numbering system. An instance for ternary numbering system is given in the next section.
\end{rem}

\section{Propositions}
In this section, examples of difference table usage are considered. These are cases of base-seven, base-eight and ternary numbering systems.
\begin{prop}
In base-seven system, $S_n \equiv j \ (mod \ 7)$ where $j \in \{ 0, 1, 3, 6 \}$ while $S_n \not\equiv \alpha \ (mod \ 7)$ where $\alpha \in \{ 2, 4, 5\}$ for any $n$.
\end{prop}

\begin{proof}
Because $7$ is an odd number, term \ref{eq:5_1} is a whole number. Therefore combinations of $i$ and $j$ are sought that ensure the wholeness of term \ref{eq:5_2}. In other words, the values of the difference table must be divisible by $7$ without a remainder.

For base-seven system, one should consider a $7X7$ difference table (Table \ref{tab:2}). As follows from the table, only values of $j \in \{ 0, 1, 3, 6 \}$ satisfy to the condition '$m$ is a whole number' and therefore $S_n \equiv j \ (mod \ 7)$ only from this set (Fig. \ref{fig2}a).
\end{proof}

\begin{table}[htb]
  \centering
  \caption{A $7X7$ difference table. Values of differences divisible by $7$ without a remainder are given on gray background. Values if $j$ satisfying the condition '$m$ is a whole number' are given in bold face.}
    \begin{tabular}{rrrrrrrrr}
    \hline
    j&&\multicolumn{7}{c}{$\cfrac{i(i +1)}{2}$}\\ \cline{3-9}
     &    & 0     & 1     & 3     & 6     & 10    & 15    & 21 \\
    \hline
    \textbf{0} &   & \cellcolor{lightgray} 0    & 1     & 3     & 6     & 10    & 15    & 21 \\
    \textbf{1} &   & -1    & 0     & 2     & 5     & 9     & \cellcolor{lightgray} 14    & 20 \\
    2 &   & -2    & -1    & 1     & 4     & 8     & 13    & 19 \\
    \textbf{3} &   & -3    & -2    & \cellcolor{lightgray} 0     & 3     & 7     & 12    & 18 \\
    4 &   & -4    & -3    & -1    & 2     & 6     & 11    & 17 \\
    5 &   & -5    & -4    & -2    & 1     & 5     & 10    & 16 \\
    \textbf{6} &   & -6    & -5    & -3    & \cellcolor{lightgray}0     & 4     & 9     & 15 \\
    \hline
    \end{tabular}%
  \label{tab:2}%
\end{table}%

\begin{table}[htb]
  \centering
  \caption{A $8X8$ difference table. Values of differences divisible by $8$ without a remainder or with a remainder $1/2$ or $-1/2$ are given on gray background. Values if $j$ satisfying the condition '$m$ is a whole number' are given in bold face.}
    \begin{tabular}{rrrrrrrrrrc}
    \hline
    j&&\multicolumn{8}{c}{$\cfrac{i(i +1)}{2}$}\\ \cline{3-10}
     &    & 0     & 1     & 3     & 6     & 10    & 15    & 21    & 28 \\
    \hline
    \textbf{0} &   & \cellcolor{lightgray}0     & 1     & 3     & 6     & 10    & 15    & 21    &\cellcolor{lightgray} 28 \\
    \textbf{1} &   & -1    & \cellcolor{lightgray}0     & 2     & 5     & 9     & 14    &\cellcolor{lightgray} 20    & 27 \\
    \textbf{2} &   & -2    & -1    & 1     & \cellcolor{lightgray}4     & \cellcolor{lightgray}8     & 13    & 19    & 26 \\
    \textbf{3} &   & -3    & -2    & \cellcolor{lightgray}0     & 3     & 7     & \cellcolor{lightgray}12    & 18    & 25 \\
    \textbf{4} &   & \cellcolor{lightgray}-4    & -3    & -1    & 2     & 6     & 11    & 17    & \cellcolor{lightgray}24 \\
    \textbf{5} &   & -5    & \cellcolor{lightgray}-4    & -2    & 1     & 5     & 10    & \cellcolor{lightgray}16    & 23 \\
    \textbf{6} &   & -6    & -5    & -3    & \cellcolor{lightgray}0     & \cellcolor{lightgray}4     & 9     & 15    & 22 \\
    \textbf{7} &   & -7    & -6    & \cellcolor{lightgray}-4    & -1    & 3     & \cellcolor{lightgray}8     & 14    & 21 \\
    \hline
    \end{tabular}%
  \label{tab:3}%
\end{table}%

\begin{prop}
In base-eight system, $S_n \equiv j \ (mod \ 8)$ where \\ $j \in \{ 0, 1, 2, 3, 4, 5, 6, 7 \}$.
\end{prop}

\begin{proof}
Because $8$ is an even number, the term \ref{eq:5_1} may be whole number (for even $k$) and a fractional number (for odd $k$). Namely, for odd $k$ $\cfrac{k(L\cdot k + (2i + 1))}{2} \equiv \cfrac{1}{2} \ (mod \ 2)$. Therefore, the differences are sought that are divisible by $8$ either without a remainder or with the remainder $1/2$ or $-1/2$.

A $8X8$ difference table should be considered (Table \ref{tab:3}). As follows from the table, all the digits of the base-eight system satisfy the condition '$m$ is a whole number', i.e. $S_n \equiv j \ (mod \ 8)$ for all the $j \in \{ 0, 1, 2, 3, 4, 5, 6, 7 \}$ (Fig. \ref{fig2}b).
\end{proof}

\begin{table}[htb]
  \centering
  \caption{A $9X9$ difference table. $j_T$ is a ternary representation of $j$ numbers. Values of differences containing $9$ are given on gray background. Tens-and-units values of $j_T$ that do not satisfy the condition '$m$ is a whole number' are given in bold italic face.}
    \begin{tabular}{rrrrrrrrrrrr}
    \hline
    j& $j_T$&&\multicolumn{9}{c}{$\cfrac{i(i +1)}{2}$}\\ \cline{4-12} \\
     &     &   & 0     & 1     & 3     & 6     & 10    & 15    & 21    & 28 & 36  \\
    \hline
    0 & 0  &   & \cellcolor{lightgray}0     & 1     & 3     & 6     & 10    & 15    & 21    & 28 & \cellcolor{lightgray}36 \\
    1 & 1  &   & -1    & \cellcolor{lightgray}0     & 2     & 5     & \cellcolor{lightgray}9     & 14    & 20    & 27 & 35 \\
    2 & 2  &   & -2    & -1    & 1     & 4     & 8     & 13    & 19    & 26 & 34 \\
    3 & 10 &   & -3    & -2    & \cellcolor{lightgray}0     & 3     & 7     & 12    & \cellcolor{lightgray}18    & 25 & 33 \\
    4 & \textbf{\textit{11}} &   & -4    & -3    & -1    & 2     & 6     & 11    & 17    & 24 & 32 \\
    5 & \textbf{\textit{12}} &   & -5    & -4    & -2    & 1     & 5     & 10    & 16    & 23 & 31 \\
    6 & 20 &   & -6    & -5    & -3    & \cellcolor{lightgray}0     & 4     & \cellcolor{lightgray}9     & 15    & 22 & 30 \\
    7 & \textbf{\textit{21}} &   & -7    & -6    & -4    & -1    & 3     & 8     & 14    & 21 & 29 \\
    8 & \textbf{\textit{22}} &   & -8    & -7    & -5    & -2    & 2     & 7     & 13    & 20 & 28 \\
    \hline
    \end{tabular}%
  \label{tab:4}%
\end{table}%

\begin{rem}
Not all even $L$ will bring about that $j$ covers all the digits of the numbering system. For example, for $L = 6$, $L = 10$ and others there will be gaps in the sets of units digits. Empirically, only the cases $L = 2^c$ ($c$ being a whole number) lead to that $j$ covers all the system digits.
\end{rem}

\begin{prop}
In ternary system, $S_n \not\equiv 11$ $(mod \ 100)$ and $S_n \not\equiv 21$ $(mod \ 100)$ for any $n$.
\end{prop}

\begin{proof}
Converting the relations to decimal system one gets $S_n \not\equiv 4$ $(mod \ 9)$ \\ and $S_n \not\equiv 7$ $(mod \ 9)$.

A $9X9$ difference table should be considered in which difference values are sought that contain $9$ (Table \ref{tab:4}). As follows from the table, neither ternary $11$ nor ternary $21$ are present among values appearing as tens-and-units digits in ternary representation of $S_n$.
\end{proof}

\section{Binomial coefficients and multicellular \\ growth}

Because of whole number and summing nature of diagonal binomial coefficients they may be applied to cell-based modeling of biological growth.

A classical example of whole number modeling is provided by a consideration of population growth of cells each of which doubles in equal time intervals producing the next cell generation. Such a dynamics gives rise to sequence of the cell numbers like $1$, $2$, $4$, $8$, $\dots$ , $2^\beta $, $\beta $ being the number of the cell generation while value $2$ in this particular case being the reproduction factor. In other words, this is an example of how local divisions of cell lead to an \textit{exponential law} of cell population growth.

Within a multicellular organism, however, the cell population growth obeys much more complicated rules. On the example of higher plant organisms the following rules (kinds of cells) may be identified. First, there are the so-called initial cells that preserve the ability to divide in the course of the entire life span of the organism. Second, the cells---immediate descendants of the initials---can divide for some time but sooner or later transform to the third kind of cells. The third kind are remote descendants that are differentiated cells that may be dead or alive but no longer divide. Because the amount of initial cells compared to  other kinds is low a modeling may 'imply' their existence but not to take into account explicitly.

Let's $M$ be the number of differentiated cells and $m$ be the number of dividing cells at the moment (generation number) $0$. The next moment (generation) the number of dividing cells are multiplied by $2$---each of them divides in two. A share of the new cells goes to the pool of differentiated cells and another share remains to be dividing. In order to get a law of the whole population growth it is necessary to suppose how the number of dividing cells alters from generation to generation. Why their number can change is a question of a separate sort and not considered here in detail. One can note, however, that the number of newly differentiated cells may influence the number of the new generation of dividing cells.

Suppose, first, that the number of dividing cells grows from generation to generation linearly, namely, by one cell a generation. Then the whole dynamics of population of $N$ cells in generations from $0$ to $q$ may look as in expressions \ref{eq:6_1}--\ref{eq:6_5}. In expression \ref{eq:6_2} (generation $1$), for example, the term $2m - (m + 1)$ denotes a new portion of differentiated cells while $(m + 1)$ is the new amount of dividing cells which is by one cell bigger as in generation $0$.

\begin{align}
M &+ m &N_0\label{eq:6_1} \\
M &+ 2m - (m + 1) + (m + 1) &N_1 \label{eq:6_2} \\
M &+ 2m - (m + 1) + 2(m + 1) - (m + 2) + (m + 2) &N_2 \label{eq:6_3} \\
M &+ 2m - (m + 1) + 2(m + 1) - (m + 2) + 2(m + 2) - & \label{eq:6_4} \\
& - (m + 3) + (m + 3) &N_3 \nonumber \\
&\dots & \nonumber \\
M &+ 2m + 2(m + 1) + 2(m + 2) + \dots + 2(m + q - 1) - &\label{eq:6_5} \\
& - [(m + 1) + (m + 2) + \dots + (m + q - 1)] &N_q \nonumber
\end{align}

In expression \ref{eq:6_5}, collecting terms and taking into account that $1$ + $2$ + $3$ + $\dots$ + $q - 1$ is $\cfrac{(q - 1)q}{2}$ give $N_q$, the size of the cell population at generation $q$ in the form
\begin{equation}
N_q = M + m + m\cdot q + \frac{(q -1)q}{2},
\label{eq:7}
\end{equation}
which is an equivalent to coefficients of the binomial theorem \ref{eq:1} truncated to the quadratic term. In other words, the idealized multicellular organism grows as minor totals of the natural sequence. A simple corollary of the dynamics is that the growth obeys a \textit{power law}, in this particular case it is a quadratic power function. It is easy to show that if the number of dividing cells stays constant then the entire organism grows linearly. If the number of dividing cells falls linearly then the entire organism undergoes a decay of growth which obeys a power law as well.

The number of dividing cells may alter in a non-linear manner. Suppose then that their number grows as the minor totals of the natural sequence, which is the same as binomial coefficients at the quadratic term. From Pascal triangle (fig. \ref{fig1}b), one can expect that the total growth should obey the pattern of binomial coefficients at the cubic term. In fact, relations between binomial coefficients at the quadratic and cubic terms, $S_n$ and $T_n$, correspondingly, are rather transparent. The summation of $S_n$ naturally gives $T_n$:

\begin{eqnarray}
T_n &=& \sum_{\kappa =1}^n S_{\kappa} = \sum_{\kappa =1}^n \frac{\kappa (\kappa +1)}{2} = \frac{1}{2}\sum_{\kappa =1}^n \kappa ^2 + \frac{1}{2}\sum_{\kappa =1}^n \kappa = \label{eq:8} \\
&=& \frac{n}{12}(n + 1)(2n + 1) + \frac{n}{4}(n + 1) = \frac{n(n + 1)(n + 2)}{2\cdot 3}\nonumber
\end{eqnarray}

\begin{align}
M &+ \frac{m(m + 1)}{2} &N_0\label{eq:9_1} \\
M &+ m(m + 1) - \frac{(m + 1)(m + 2)}{2} + \frac{(m + 1)(m + 2)}{2} &N_1 \label{eq:9_2} \\
M &+ m(m + 1) - \frac{(m + 1)(m + 2)}{2} + (m + 1)(m + 2) - & \label{eq:9_3} \\
& - \frac{(m + 2)(m + 3)}{2} + \frac{(m + 2)(m + 3)}{2} &N_2 \nonumber \\
M &+ m(m + 1) - \frac{(m + 1)(m + 2)}{2} + (m + 1)(m + 2) - & \label{eq:9_4} \\
& - \frac{(m + 2)(m + 3)}{2} + (m + 2)(m + 3) - \nonumber \\
& - \frac{(m + 3)(m + 4)}{2} + \frac{(m + 3)(m + 4)}{2} &N_3 \nonumber \\
&\dots & \nonumber \\
M &+ m(m + 1) + (m + 1)(m + 2) + (m + 2)(m + 3) + \dots &\label{eq:9_5}\\
& + (m + q - 1)(m + q) + \frac{(m + q)(m + q + 1)}{2} - & \nonumber \\
& - \bigg(\frac{(m + 1)(m + 2)}{2} + \frac{(m + 2)(m + 3)}{2} + \dots &\nonumber \\
& + \frac{(m + q - 1)(m + q)}{2} + \frac{(m + q)(m + q + 1)}{2}\bigg) &N_q \nonumber
\end{align}

Expressions \ref{eq:9_1}--\ref{eq:9_5} show the whole cell population dynamics in detail implying that the cell population consists of $M$ differentiated and $\cfrac{m(m + 1)}{2}$ dividing cells at the generation $0$. Analogically to expression \ref{eq:6_2}, in \ref{eq:9_2}, $m(m + 1) - \cfrac{(m + 1)(m + 2)}{2}$ is the new differentiated cell while \\ $\cfrac{(m + 1)(m + 2)}{2}$ is the new number of dividing cells.

Collecting terms in expression \ref{eq:9_5} one gets
\begin{eqnarray}
N_q = M &+& m(m + 1) + \frac{(m + 1)(m + 2)}{2} \label{eq:10} \\
&+& \frac{(m + 2)(m + 3)}{2} + \dots + \frac{(m + q - 1)(m + q)}{2}, \nonumber
\end{eqnarray}
which is with the help of equations \ref{eq:2} and \ref{eq:8} transferred to
\begin{eqnarray}
N_q &=& M + S_m + S_m + S_{(m+1)} + S_{(m + 2)} + \dots + S_{m+q-1} = \label{eq:11} \\
&=& M + S_m + T_{m+q-1} - T_{m-1}. \nonumber
\end{eqnarray}

In equation \ref{eq:11}, the terms $M + S_m - T_{m-1}$ is a mere constant depending on initial conditions ($M$ and $m$). The term $T_{m+q-1}$ gives dynamics of the number of cells in the idealized multicellular organism. Obviously, $T_{m+q-1}$ is a polynomial of the third order.

An important implication of the cell population behavior modeled with diagonal binomial coefficient arithmetics is that it is a power law that governs growth of a multicellular organism, naturally, in terms of cell number. In mathematical modeling of biological objects, the use of polynomial forms at approximating of various biological data has been called into question. While providing a high approximation accuracy the polynomial form often lack profound justifications and interpretability of parameters. The approach presented in the above analysis may help to fill the gap in polynomial usage in biological growth modeling.

\end{document}